\documentclass[a4paper,12pt]{amsart}

\usepackage{version}
\excludeversion{draft}
\excludeversion{mathias}
\excludeversion{xia}

\usepackage[english]{babel}
\usepackage{amssymb,amsfonts,amsxtra,amsmath}
\usepackage{dsfont,mathrsfs}
\renewcommand{\mathcal}{\mathscr}
\usepackage[all]{xypic}
\xyoption{dvips}
\usepackage{url}
\usepackage{fullpage}
\usepackage[colorlinks,backref]{hyperref}
\usepackage[neveradjust]{paralist}
\usepackage{mathtools}

\DeclarePairedDelimiter{\floor}{\lfloor}{\rfloor}

\theoremstyle{definition}
\newtheorem{ntn}{Notation}[section]
\newtheorem{dfn}[ntn]{Definition}
\theoremstyle{plain}
\newtheorem{lem}[ntn]{Lemma}
\newtheorem{prp}[ntn]{Proposition}
\newtheorem{thm}[ntn]{Theorem}
\newtheorem{cor}[ntn]{Corollary}

\theoremstyle{remark}

\newtheorem{rmk}[ntn]{Remark}

\numberwithin{equation}{section}

\newcommand{\ideal}[1]{{\left\langle#1\right\rangle}}
\newcommand{\xymat}{\SelectTips{cm}{}\xymatrix}
\newcommand{\into}{\hookrightarrow}

\newcommand{\p}{\partial}

\newcommand{\wt}{\widetilde}

\newcommand{\A}{\mathcal{A}}
\renewcommand{\AA}{\mathbb{A}}

\newcommand{\E}{\mathcal{E}}
\newcommand{\F}{\mathcal{F}}

\newcommand{\GG}{\mathds{G}}
\newcommand{\G}{\mathcal{G}}
\newcommand{\I}{\mathcal{I}}
\newcommand{\J}{\mathcal{J}}
\newcommand{\K}{\mathcal{K}}
\renewcommand{\L}{\mathcal{L}}
\newcommand{\M}{\mathcal{M}}
\newcommand{\NN}{\mathbb{N}}
\renewcommand{\O}{\mathcal{O}}
\newcommand{\PP}{\mathds{P}}
\newcommand{\QQ}{\mathds{Q}}
\newcommand{\reg}{\mathrm{reg}}
\newcommand{\red}{\mathrm{red}}
\newcommand{\sing}{\mathrm{sing}}
\newcommand{\ZZ}{\mathds{Z}}
\newcommand{\mm}{\mathfrak{m}}
\newcommand{\pp}{\mathfrak{p}}
\newcommand{\qq}{\mathfrak{q}}

\DeclareMathOperator{\depth}{depth}
\DeclareMathOperator{\divisor}{div}
\DeclareMathOperator{\grade}{grade}
\DeclareMathOperator{\codim}{codim}
\DeclareMathOperator{\ch}{ch}
\DeclareMathOperator{\Der}{Der}
\DeclareMathOperator{\SDer}{\mathcal{D}\kern -.5pt er}

\DeclareMathOperator{\Ext}{Ext}
\DeclareMathOperator{\SExt}{\mathcal{E}\kern -.5pt xt}
\DeclareMathOperator{\Hom}{Hom}
\DeclareMathOperator{\SHom}{\mathcal{H}\kern -.5pt om}
\DeclareMathOperator{\SH}{\mathcal{H}}
\DeclareMathOperator{\id}{id}
\DeclareMathOperator{\ord}{ord}
\DeclareMathOperator{\Proj}{Proj}

\DeclareMathOperator{\Spec}{Spec}

\begin{document}

\title[Derivations on graded algebras]{A cohomological interpretation of\\derivations on graded algebras}

\author[X.~Liao]{Xia Liao}
\address{X.~Liao\\
Department of Mathematics\\
University of Kaiserslautern\\
67663 Kaiserslautern\\
Germany}
\email{\href{mailto:liao@mathematik.uni-kl.de}{liao@mathematik.uni-kl.de}}


\author[M.~Schulze]{Mathias Schulze}
\address{M.~Schulze\\
Department of Mathematics\\
University of Kaiserslautern\\
67663 Kaiserslautern\\
Germany}
\email{\href{mailto:mschulze@mathematik.uni-kl.de}{mschulze@mathematik.uni-kl.de}}

\thanks{The research leading to these results has received funding from the People Programme (Marie Curie Actions) of the European Union's Seventh Framework Programme (FP7/2007-2013) under REA grant agreement n\textsuperscript{o} PCIG12-GA-2012-334355.}

\date{\today}



\begin{abstract}
We trace derivations through Demazure's correspondence between a finitely generated positively graded normal $k$-algebras $A$ and normal projective $k$-varieties $X$ equipped with an ample $\QQ$-Cartier $\QQ$-divisor $D$.
We obtain a generalized Euler sequence involving a sheaf on $X$ whose space of global sections consists of all homogeneous $k$-linear derivations of $A$ and a sheaf of logarithmic derivations on $X$.
\end{abstract}

\maketitle
\tableofcontents

\section{Introduction}

As an elementary motivation for the topic of this article, we recall the Euler sequence on projective space $X=\PP^n_k$ over a field $k$.
Setting $\M:=\O_X(1)^{n+1}$, it reads
\begin{equation}\label{20}
0 \to \O_X \to\M \to \Theta_X \to 0.
\end{equation}
For any $d\in\ZZ$, $H^0(X,\M\otimes\O_X(d))$ computes the homogeneous $k$-linear derivations of degree $d$ of the section ring $A=\bigoplus_{d\geq 0}H^0(X,\O_X(d))$.
In fact, the section ring identifies with the polynomial ring $A=k[x_0,\dots,x_n]$ and 
the module of $k$-linear derivations of $A$,
\[
\Der_k(A)\cong\bigoplus_{i=0}^nA\p_{x_i}\cong A(1)^{n+1},
\]
is free of rank $n+1$.
Therefore, $\wt{\Der_k(A)}\cong\M$ and $H^0(X,\M\otimes\O_X(d))\cong\Der_k(A)_d$. 

There is the following generalization due to Wahl~\cite{Wah83}.
Let $k$ be algebraically closed of characteristic $0$, let $X$ be a normal projective $k$-variety, and let $\L\cong\O_X(D)$ be the sheaf of sections of an ample line bundle $L$ on X. 
Then
\begin{itemize}
\item there exists a coherent $\O_X$-module $\M$ and a short exact sequence \eqref{20},
\item $\pi^{*}\M \cong \SDer_{L^{-1}}(-\log X)$, where $\pi:L^{-1}\to X$ is the projection, and $X$ is identified with the zero section of the line bundle $L^{-1}$,
\item $H^{0}(X,\M\otimes \L^{d})$ computes the homogeneous $k$-linear derivations of degree $d$ of the graded normal section ring $\bigoplus_{i\geq 0}H^{0}(X,\L^{i})$, and,
\item if $\dim X\ge2$ and $d<0$, $H^{0}(X,\M\otimes\L^{d})\cong H^{0}(X,\Theta_X\otimes\L^d)$.
\end{itemize}

Based on this result and a generalization of Zariski's lemma due to Mori and Sumihiro, Wahl gave a cohomological characterization of projective space:
If $X$ is smooth and $H^0(X,\Theta_X\otimes\L^{-1})\ne0$, then $(X,\L)\cong(\PP^n_k,\O_X(1))$ or $(X,\L)\cong(\PP^1_k,\O_X(2))$. 
From his investigations, Wahl derived the conjecture~\cite[Conj.~1.4]{Wah83b} that normal graded isolated singularities do not admit negative degree derivations, at least for \emph{some} choice of grading.
In case of isolated complete intersection singularities (ICIS), where the grading is essentially unique, Aleksandrov~\cite[\S6]{Ale85} studied this conjecture.
His results are affirmative if at least one defining equation has multiplicity at least $3$, but he also gives counter-examples. 

\smallskip

In this article, we give a cohomological interpretation of homogeneous derivations on an arbitrary finitely generated positively graded normal algebra over an arbitrary field $k$. 
Our approach relies on a construction due to Demazure reviewed in \S\ref{13}:
To a $k$-algebra $A$ under consideration, one associates an ample $\QQ$-Cartier $\QQ$-divisor $D$ on $X=\Proj A$, unique up to adding a principal divisor, such that $\O_X(iD)\cong\O_X(i)$ (see \eqref{90}) for any $i\in\ZZ$. 
Conversely, given a normal projective $k$-variety $X$ equipped with an ample $\QQ$-Cartier $\QQ$-divisor $D$, 
the section ring $A=\bigoplus_{i\in\NN}H^0(X,\O_X(iD))$ is a finitely generated positively graded normal $k$-algebra.
These two constructions are mutually inverse, hence any finitely generated positively graded normal $k$-algebra $A$ takes the form $\bigoplus_{i\in\NN}H^0(X,\O_X(iD))$ for an appropriate $\QQ$-Cartier $\QQ$-divisor $D$ on $X=\Proj A$.
Watanabe~\cite{Wat81} described the depth and the Cohen--Macaulay and Gorenstein properties of $A$ in terms of $D$.

\begin{thm}[Generalized Euler Sequence]\label{70}
Let $X$ be a normal projective $k$-variety with singular locus $Z:=X^\sing$, $D$ an ample $\QQ$-Cartier $\QQ$-divisor and $A$ the section ring.
\begin{asparaenum}
\item\label{70a} For any $d\in\ZZ$, there is a coherent reflexive $\O_X$-module $\M_{D,d}$ (see \eqref{72}) on $X$ and divisors $W_{D,d}$ and $L_{D,d}$ (see Definition~\ref{12}). 
They fit into an exact sequence
\begin{equation}\label{73}
\xymat@C=10pt{
0\ar[r]&\O_X(dD+W_{D,d})\ar[r]^-\varphi&\M_{D,d}\ar[r]^-\psi&\SDer_X(dD-\log L_{D,d})\ar[r]&\SH_Z^2(\O_X(dD+W_{D,d})),
}
\end{equation}
where $\SH_Z^\bullet$ is Grothendieck's local cohomology sheaf with support $Z$ and $\SDer_X(dD-\log L_{D,d})$ is a certain sheaf of logarithmic derivations (see Definitions~\ref{29} and \ref{12}).
\item\label{70c} If $\dim X\ge 1$, then the space of global sections $H^0(X,\M_{D,d})$ identifies with the space $\Der_k(A)_d$ of homogeneous degree-$d$ $k$-linear derivations on $A$.
\end{asparaenum}
\end{thm}

\begin{proof}
Proposition~\ref{23} proves \eqref{70a} and Lemma~\ref{39} proves \eqref{70c}.
\end{proof}

\begin{rmk}
\begin{asparaenum}

\item The hypotheses for \S\ref{71} show that $Z$ in Theorem~\ref{70} may be replaced by the locus where some irreducible component of $D$ fails to be Cartier.
However, the local cohomology sheaf in \eqref{73} is only a coarse estimate for the image of $\psi$ in general.

\item In case $A$ is a graded normal surface singularity, the sequence \eqref{73} was described by Wahl \cite[(4.3)~Thm.]{Wah88}.

\item In case $D$ is a Cartier (integer) divisor, the sequence \eqref{73} reduces to the short exact Euler sequence of Wahl (see \cite[Thm.~1.3]{Wah83}) which takes the form \eqref{20}.
Alternatively it can be obtained by adapting our arguments in \S\ref{71} (see Remark~\ref{84}).
Wahl proves surjectivity of $\psi$ using a smoothness argument limited to the case where $D$ is a Cartier divisor (see Remark~\ref{87}).

\item A particular case of the sequence \eqref{73} can be obtained by dualizing the sequence in the proof of \cite[(8.9) Lem.]{Fle81a}.
There is also short exact Euler sequence for toric varieties (see \cite[Thm.~8.1.6]{CLJS11}).

\item There is a vast generalization of the Demazure correspondence due to Altmann, Hausen and S\"uss~\cite{AH06,AHS08}.

\end{asparaenum}
\end{rmk}

\begin{cor}\label{105}
Assume that $\ch k=0$, that $A$ is $(S_3)$, and that $X$ is $(R_2)$ of dimension $\dim X\ge2$.
Then there is a short exact sequence
\begin{equation}\label{95}
\xymat{
0\ar[r]&\O_X(dD)\ar[r]&\M_{D,d}\ar[r]&\SDer_X(dD-\log L_{D,d})\ar[r]&0
}
\end{equation}
whose sequence of global sections is short exact as well. 
In particular, if $d<0$, then 
\[
\Der_k(A)_d\cong H^0(X,\M_{D,d})\cong H^0(X,\SDer_X(dD-\log L_{D,d})).
\]
\end{cor}

\begin{proof}
The proof is given in Proposition~\ref{19}.
\end{proof}

\begin{draft}

\begin{cor}\label{106}
In addition to the hypotheses of Corollary~\ref{105}, assume that $k$ is perfect and that $A$ has an isolated singularity. 
Then $T^1_{A/k,d}\subset H^1(X,\SDer_X(dD-\log L_{D,d}))$.\qed.
\end{cor}

\begin{proof}
This is proved along the lines of \cite[Proof of Thm.~5.4]{Har10}.
\end{proof}

\end{draft}

\begin{cor}\label{107}
If $X$ is regular, $W_{D,d}=0$, and $L_{D,d}$ a free divisor in the sense of Saito \cite{Sai80}, then the extension class of the short exact sequence \eqref{95} is given by $\varepsilon_D\in H^1(X,\Omega_X^1(\log L_{D,d}))$ (see \eqref{93}).
\end{cor}

\begin{proof}
The proof is given in Proposition~\ref{94}.
\end{proof}

We present two examples:

In \S\ref{38a}, we consider an ICIS counter-example to Wahl's conjecture found by Michel Granger and the second author (see \cite{GS14}).
We describe the associated  projective variety $X$ and its singularities and apply our results to interpret a degree-$-1$ derivation as a twisted vector field on $X$.

In \S\ref{38b}, we study the degree-$-1$ derivations on $k[x,y,z]/\ideal{xy-z^2}$ with $x,y,z$ all of degree $1$. 
Notice that if $\ch(k)\neq 2$ there is no such derivation, while if $\ch(k)=2$, $\p_z$ is one such derivation. 
We will give a cohomological explanation of this phenomenon.

\subsection*{Acknowledgements}

We are grateful to Jonathan Wahl for stimulating questions and useful hints.

\section{Demazure's construction}\label{13}

Demazure~\cite{Dem88} established a correspondence between 
\begin{enumerate}[(A)]
\item\label{11a} finitely generated positively graded normal algebras $A$ over a field $k$ and 
\item\label{11b} normal projective $k$-varieties $X$ equipped with an ample $\QQ$-Cartier $\QQ$-divisor defined up to ($\ZZ$-)linear equivalence.
\end{enumerate}

In the following, we briefly summarize the main features of this correspondence.

\begin{asparaenum}
\item[\eqref{11b}$\leadsto$\eqref{11a}.]

We start with an ample $\QQ$-Cartier $\QQ$-divisor $D$ on a normal projective $k$-variety $X$. 
More specifically, this condition means that we can write $D$ as a finite sum over prime Weil divisors $V$ (codimension-$1$ subvarieties of $X$)
\begin{equation}\label{88}
D=\sum_Vr_V\cdot V,\quad r_V=\frac{p_V}{q_V}\in\QQ,
\end{equation}
where $p_V$ and $q_V>0$ are coprime integers for each $V$, and there exists a positive integer $N$ such that $ND$ is an ample Cartier (integer) divisor.
For any $k$-variety $X$, let $K(X)$ denote the space of rational functions and $\K_X$ the corresponding constant sheaf.
With $\floor{-}$ being the coefficient-wise round-down, one defines
\begin{equation}\label{90}
\O_X(D):=\O_X(\floor{D})
\end{equation}
in the usual way as a coherent reflexive $\O_X$-submodule of $\K_X$.
Consider the following quasicoherent graded $\O_X$-algebra
\begin{equation}\label{62}
\A:=\bigoplus_{i\in\ZZ}\A_i\supset\bigoplus_{i\in\NN}\A_i=:\A^+.
\end{equation}
defined by
\[
\A_i:=\O_X(iD)T^i.
\]
The inclusion \eqref{62} yields an open embedding of cylinders
\begin{equation}\label{63}
\xymat{
C:=\Spec_X\A\ar@{^(->}[r]&\Spec_X\A^+=:C^+\ar[r]^-\pi&X
}
\end{equation}
over $X$ whose complement is the image of the zero-section of $C^+\to X$ defined by $\A^+\to\A_0$, which is isomorphic to $X$ (see \cite[2.2~Lem.]{Dem88}).

\begin{rmk}\label{87}
In case $D$ is a Cartier (integer) divisor, $\O_X(D)$ is the sheaf of sections of an ample line bundle $L$ on $X$.
The cylinder $C^+$ then becomes the total space of the line bundle $L^{-1}$ (see \cite[(1.2)]{Wah83}) and the morphism $\pi$ is smooth.
This was used by Wahl to construct the short exact sequence \eqref{20} in this case (see \cite[Proof of Thm.~1.3]{Wah83}).

In general $\pi$ is smooth if and only if it is flat with geometrically regular fibers.
By \cite[2.8~Prop.]{Dem88}, the latter enforces $q_V=1$ in \eqref{88}.
Moreover, flatness of $\A$ over $\O_X$ means that $\A_i\cong\O_X(iD)$ is locally free for all $i\in\ZZ$.
Therefore smoothness of $\pi$ occurs exactly in the case $D$ is a Cartier (integer) divisor.
\end{rmk}

The desired positively graded $k$-algebra associated with $X$ and $D$ is
\[
A:=H^0(X,\A^+)=\bigoplus_{i\in\NN}A_i,\quad A_i:=H^0(X,\O_X(iD))T^i.
\]
The restriction morphisms $A\to H^0(U,\A^+)$ glue to a canonical morphism
\begin{equation}\label{10}
C^+ \to \Spec A=:\Gamma^+,
\end{equation}
contracting the image of the zero-section of $C^+\to X$ to the vertex $\mm$ of the cone $\Gamma^+$.
It induces an isomorphism 
\begin{equation}\label{27}
C\cong\Gamma:=\Spec A\setminus\{\mm\}
\end{equation}
of the open cylinder with the open cone and, by passing to the $\GG_m$-quotient, an isomorphism
\[
X\cong\Proj A, 
\]
where $A$ is finitely generated (see \cite[3.2~Prop., 3.3~Prop.]{Dem88}).
In particular, this gives an identification (see \cite[Cor.(3.2)]{Dem88}).
\begin{equation}\label{51}
K(C)=K(X)(T)=Q(A).
\end{equation}
The normality of $\Gamma^+$ and hence $A$ follows using \eqref{27} from \cite[2.7~Lem.]{Dem88} and \cite[Cor.~(2.3)]{Wat81}, together with Serre's normality criterion.
Alternatively, it can be verified directly by checking that $\A$ is a sheaf of integrally closed rings.

\item[\eqref{11a}$\leadsto$\eqref{11b}.]

Assume now conversely given a finitely generated positively graded normal $k$-algebra $A$.
Then one obtains a normal projective $k$-variety by setting
\[
X:=\Proj A.
\] 
Without loss of generality, we may assume that the degrees of the generators are coprime.
For $i\in\ZZ$, denote as usual by 
\[
\O_X(i):=\wt{A(i)}
\]
the coherent sheaf on $X$ associated with the graded module $A(i)$. 
Taking $\A_i:=\O_X(i)$, one obtains quasicoherent graded $\O_X$-algebras $\A\supset\A^+$ and cylinders $C\subset C^+$ as in \eqref{62} and \eqref{63}.
By assumption on the degrees of the generators of $A$, there is a homogeneous rational function $T\in Q(A)$ of degree $1$. 
By the way $C$ is defined, $T\in K(C)$ and the homogeneity of $T$ means that the associated principal divisor $\divisor(T)$ on $C$ is $\GG_m$-stable.
It then comes from a $\QQ$-divisor $D$ on $X$ (see \cite[3.5~Thm.]{Dem88}),
\[
\divisor(T)=\pi^*(D).
\]
\end{asparaenum}
A different choice of $T$, say $T'$, will result in a different divisor $D$, say $D'$, on $X$. 
However, $T/T'\in K(X)(T)$, being homogeneous of degree $0$, implies that $T/T'$ defines an $f\in K(X)$.
Thus, $D$ and $D'$ differ by a principal divisor,
\begin{equation}\label{65}
D=D'+\divisor(f).
\end{equation}
From the normality of $A$, one deduces the following equalities (see \cite[3.5~Thm.]{Dem88})
\begin{equation}\label{64}
A_i=H^0(X,\O_X(iD))T^i,\quad\O_X(i)=\O_X(iD)T^i,\quad\text{for all }i\in\ZZ,
\end{equation}
of subspaces of $K(C)$ and subsheaves of $\pi_*\K_C$ respectively.
We denote the $k$-th Veronese subring of $A$ by 
\[
A^{(k)}:=\bigoplus_{i\in\NN}A_{ik}.
\]
Since $A$ is finitely generated there exists an $N\ge1$ such that, for all $i\ge1$, $A^{(iN)}$ is generated over $A_0$ by $A^{(iN)}_1=A_{iN}$ (see \cite[Ch.~III, \S1, Prop.~3]{Bou98}).
Combined with the second equality in \eqref{64}, this shows that $iND$ is a very ample divisor for those $i\ge1$ which clear the denominators of $ND$.

\section{Reflexive sheaves and rational derivations}\label{69}

We now discuss the basic idea of this paper. 
We adopt the setup of Demazure's construction (see \S\ref{13}) and use the following list of data freely in the sequel:
\begin{itemize}
\item a finitely generated positively graded normal $k$-algebra $A$ with coprime degrees of generators,
\item a corresponding ample $\QQ$-Cartier $\QQ$-divisor $D$ on $X=\Proj A$,
\item a homogeneous $T\in K(C)$ of degree $1$ with $\divisor(T) = \pi^{-1}(D)$,
\item the graded $\O_X$-algebra $\A=\bigoplus_{i\in\ZZ}\O_X(iD)T^{i}=\bigoplus_{i\in\ZZ}\O_X(i)$,
\item the cylinder $\pi:C=\Spec_X(\A)\to X$ over $X$ and the cone $\Gamma^+=\Spec A$.
\end{itemize}

Recall that a coherent $\O_X$-module $\F$ is called reflexive if the canonical map
\[
\F\to\F^{\vee\vee}
\]
to the double-dual is an isomorphism.
In other words, reflexive $\O_X$-modules are duals of coherent $\O_X$-modules (see \cite[Cor.~1.2]{Har80}).
In particular, the tangent sheaf
\begin{equation}\label{28}
\Theta_X:=\SDer_k\O_X=(\Omega^1_{X/k})^\vee
\end{equation}
is a reflexive $\O_X$-module for any normal $k$-variety $X$.
An $\O_X$-module $\F$ is called normal if the restriction map
\[
\F(U)\to\F(U\setminus Y)
\]
is bijective for all $U\subset X$ open and $Y\subset U$ closed of codimension $\ge2$.

\begin{lem}[{\cite[Prop.~1.6]{Har80}}]\label{30}
A coherent $\O_X$-module $\F$ is reflexive if and only if it is torsion-free and normal.\qed
\end{lem}

The grading on $\pi_*\O_C=\bigoplus_{i\in\ZZ}\O_X(iD)T^{i}$ induces a grading on the $\O_X$-module
\begin{equation}\label{72}
\M_D:=\pi_*\Theta_C=\SDer_k\pi_*\O_C=\bigoplus_{d\in\ZZ}\M_{D,d},\quad\M_{D,d}=(\pi_*\Theta_C)_d.
\end{equation}

The following two lemmas establish the properties of $\M_{D,d}$ stated in Theorem~\ref{70}.

\begin{lem}\label{24}
The sheaves $(\pi_*\Theta_C)_d$ are coherent reflexive $\O_X$-modules.
\end{lem}

\begin{proof}
Since $\Theta_C$ is a coherent $\O_C$-module and $\pi_*$ is exact ($\pi$ is an affine morphism), $\pi_*\Theta_C$ is a coherent $\pi_*\O_C$-module.
It follows that $(\pi_*\Theta_C)_d$ is coherent over $(\pi_*\O_C)_0=\O_X$.

Let $y\in X$ be an arbitrary (closed or non-closed) point. 
By \cite[2.5]{Dem88}, $\pi^{-1}(y)^{\red}\cong\Spec\kappa(y)[t,t^{-1}]$ is $1$-dimensional.
In particular, this implies that $\codim_C\pi^{-1}(Y) = \codim_XY$ for any closed subset $Y\subset X$.
Thus, reflexivity of $(\pi_*\Theta_C)_d$ follows from reflexivity of $\Theta_C$ (see \ref{28}) using Lemma~\ref{30}.
\end{proof}

\begin{lem}\label{39}
If $\dim X\ge1$, then $H^0(X,(\pi_*\Theta_C)_d)=(\Der_kA)_d$.
\end{lem}

\begin{proof}\pushQED{\qed}
Since $A$ is a normal $k$-algebra, $\Gamma^+$ is normal affine $k$-variety (see \eqref{10}) and $\codim_{\Gamma^+}\{\mm\}=\dim(\Gamma^+)=\dim X+1\ge 2$ where $\mm$ is the vertex of $\Gamma^+$.
By \eqref{28} and Lemma~\ref{30}, $\Theta_{\Gamma^+}$ is normal and, using \eqref{27}, there is a graded identification
\[
\Der_kA=H^0(\Gamma^+,\Theta_{\Gamma^+})=H^0(\Gamma,\Theta_{\Gamma})=H^0(C,\Theta_C)=H^0(X,\pi_*\Theta_C).\qedhere
\]
\end{proof}

The following lemma serves to define the maps in the generalized Euler sequence \eqref{73} in Theorem~\ref{70}.

\begin{lem}\label{66}
There is a canonically split short exact sequence
\begin{equation}\label{53}
\xymat@R=1pt{
0\ar[r]&\K_C\ar[r]^-\Phi&\SDer_k\K_C\ar[r]^-\Psi&\ar[r]\pi^*\SDer_k\K_X\ar[r]&0\\
&1\ar@{|->}[r]&T\p_T
}
\end{equation}
obtained as the localization of an underlying graded sequence.
\end{lem}

\begin{proof}
Since $T$ in \eqref{51} is transcendental over $K(X)$, $K(X)\subset K(C)$ a separable field extension.
This yields a short exact sequence of $K(C)$-vector spaces (see \cite[Thm.~26.6]{Mat89})
\[
0\gets\Omega_{K(C)/K(X)}\gets\Omega_{K(C)/k}\gets K(C)\otimes_{K(X)}\Omega_{K(X)/k}\gets0.
\]
Its $K(C)$-dual reads
\begin{equation}\label{104}
0\to\Der_{K(X)}(K(C))\to\Der_k(K(C))\to K(C)\otimes_{K(X)}\Der_k(K(X))\to 0.
\end{equation}

Since $\Omega^1_{K(X)[T]/K(X)}=K(X)[T]dT$ and K\"ahler differentials commute with localization,
\begin{align*}
\Der_{K(X)}(K(C))&=\Hom_{K(C)}(\Omega^1_{K(C)/K(X)},K(C))\\
&=\Hom_{K(C)}(K(C)\otimes_{K(X)[T^{\pm1}]}\Omega^1_{K(X)[T^{\pm1}]/K(X)},K(C))\\
&=K(C)\otimes_{K(X)[T^{\pm1}]}\Hom_{K(C)}(\Omega^1_{K(X)[T^{\pm1}]/K(X)},K(X)[T^{\pm1}])\\
&=K(C)\otimes_{K(X)[T^{\pm1}]}\Der_{K(X)}(K(X)[T^{\pm1}])
\end{align*}
is a localization of the graded module $\Der_{K(X)}(K(X)[T^{\pm1}])$ which is free with with homogeneous basis $T\partial_T$ of degree $0$.
Similarly the entire sequence~\eqref{104} is obtained by localizing the dual of the graded short exact sequence
\begin{equation}\label{113}
0\gets\Omega_{K(X)[T]/K(X)}\gets\Omega_{K(X)[T]/k}\gets K(X)[T]\otimes_{K(X)}\Omega_{K(X)/k}.
\end{equation}
Viewing all terms of the sequence~\eqref{104} as constant sheaves yields the sequence~\eqref{53}.

The canonical splitting of the sequence~\eqref{104} is induced the composition of the inclusions
\[
\Der_k(K(X))\into\Der_k(K(X)[T])\into\Der_k(K(C)),
\]
the first one defined by setting $\sigma(T)=0$ for any $\sigma\in\Der_k(K(X))$, and the second one by extension to the fraction field.
\end{proof}

Now we are ready to explain the idea behind Theorem~\ref{70}:
By Lemma~\ref{39} and \eqref{72}, any $\delta\in(\Der_k A)_d$ can be considered as a rational derivation $\delta\in H^0(X,\SDer_k\pi_*\K_C)=K(C)$ with the additional requirement that
\begin{equation}\label{74}
\delta(\O_X(iD)T^i)\subset\O_X((i+d)D)T^{i+d}.
\end{equation}
More explicitly, for a dense open subset $U$ and its (dense open) preimage $V:=\pi^{-1}(U)$, the restriction of $\delta$ to $U$ in $H^0(U,\pi_*(\Theta_C)_d)\subset H^0(V,\Theta_C)$ can be viewed as a rational derivation $\delta\in\Der_k(K(C))$.
By Lemma~\ref{66}, $\delta$ decomposes uniquely as
\begin{equation}\label{52}
\delta=\sigma+\alpha T\partial_T,
\end{equation}
where $\alpha \in K(C)$ and $\sigma\in K(C)\otimes_{K(X)}\Der_k(K(X))$.
The condition~\eqref{74} imposes restrictions on the possible choices of $\sigma$ and $\alpha$ in \eqref{52}. 
We will investigate these options via a local computation in \S\ref{71}. 
The first two arrows in our generalized Euler sequence in Theorem~\ref{70} will be induced by the simple minded inclusion $\alpha\mapsto\alpha T^dT\partial_T=\pi_*\Phi(\alpha T^d)$ and the projection $\delta\mapsto T^{-d}\sigma=\pi_*\Psi(T^{-d}\delta)$ induced by $\pi_*$ of \eqref{53} (see \eqref{75} and \eqref{76}).
The possible $\sigma$ in \eqref{52} will turn out to form a sheaf of certain logarithmic derivations.
It is defined in terms of rational derivations in the spirit of Zariski (see \cite[Ch.\,I]{Kni73}).

\begin{dfn}\label{29}
Let $S$ be a $\QQ$-divisor and $L$ a reduced effective Weil divisor.
Define the sheaf of derivations with poles along $S$ logarithmic along $L$ as the sheaf sections of the constant sheaf $\SDer_k\K_X$ of rational $k$-linear derivations on $X$ which map $\O_X(-L)$ to $\O_X(S-L)$ (see \eqref{90}),
\[
\SDer_X(S-\log L):=\O_X(S-L):_{\SDer_k\K_X}\O_X(-L).
\]
\end{dfn}

\begin{rmk}\label{101}
For $S=0$ in Definition~\ref{29}, the sheaf $\SDer_X(S-\log L)=\SDer_X(-\log L)$ is the one defined by Saito~\cite{Sai80}, as suggested by our notation.
However for a (possibly non-reduced) effective divisor $S$ it differs from the multilogarithmic derivations associated with hyperplane arrangements as defined by Ziegler (see \cite{Zie89}).
\end{rmk}

\begin{lem}\label{34}
\begin{asparaenum}
\item\label{34a} The sheaf $\SDer_X(S-\log L)$ is a coherent reflexive $\O_X$-module.
\item\label{34b} The canonical map 
\[
\SDer_X(-\log L)\otimes_{\O_X}\O_X(S)\to\SDer_X(S-\log L)
\]
is an isomorphism if $\floor{S}$ is a Cartier divisor.
\item\label{34c} The map in \eqref{34b} induces an isomorphism
\[
(\SDer_X(-\log L)\otimes_{\O_X}\O_X(S))^{\vee\vee}\cong\SDer_X(S-\log L).
\]
\end{asparaenum}
\end{lem}

\begin{proof}
The statements are of local nature, so we may assume that $X=\Spec B$ is affine.
Then $H^0(X,\O_X(S-L))$ is a fractional ideal.
It follows that, for any non-zero divisor $g\in H^0(X,\O_X(-L))$, there is an $f\in B$ such that
\begin{equation}\label{103}
\I:=\O_X(S-\divisor(f)-L)\subset\O_X(-\divisor(g))=\O_X\cdot g\subset\O_X(-L)=:\J
\end{equation}
is a chain of inclusions of (coherent) ideal sheaves.
As derivations extend uniquely to fields of fractions, $\SDer_k(\K_X)=\SDer_k(\O_X,\K_X)$.
Since multiplication by $f$ defines an automorphism of the latter $\O_X$-module, we obtain
\begin{equation}\label{82}
f\cdot\SDer_X(S-\log L)\cong\SDer_X(S-\divisor(f)-\log L).
\end{equation}
After suitably modifying $S$, we may therefore assume that $f=0$.
By the Leibniz rule, $\SDer_k(\O_X\cdot g)\subset\SDer_k(\O_X)$ and \eqref{103} implies that
\[
\SDer_X(S-\log L)\cong\I:_{\SDer_k\O_X}\J.
\]
Writing $\J=\ideal{f_1,\dots,f_m}_{\O_X}$, another application of the Leibniz rule shows that
\[
\I:_{\SDer_k\O_X}\J=\SDer_k(\O_X,\I:_{\O_X}\J)\cap\bigcap_{i=1}^m\I:_{\SDer_k\O_X}f_i.
\]
Thus, coherence of $\SDer_X(S-\log L)$ follows from coherence of $\SDer_k(\O_X,\M)=\SHom_{\O_X}(\Omega^1_X,\M)$ for any coherent $\O_X$-module $\M$.
Since $\SDer_X(S-\log L)$ is defined by codimension-$1$ conditions, it is a normal sheaf and hence reflexive by Lemma~\ref{30}.
This proves \eqref{34a}, and \eqref{34b} follows from \eqref{82} taking $\divisor(f):=\floor{S}$.
To see \eqref{34c}, use that $\floor{S}$ is Cartier in codimension $1$ and apply \eqref{34a} and Lemma~\ref{30}.
\end{proof}

\section{Local derivations on a cylinder}\label{71}

In this section, we investigate the condition \eqref{74}.
Let $D$ be a $\QQ$-divisor on $X$ and write it as in \eqref{88}. 
Let $U=\Spec B$ be an open affine subset of $X$ on which all the $V$ appearing \eqref{88} are Cartier. 
This condition is met, for example, when $U$ is regular.  

We first treat the case where $D$ has only one irreducible component. 
Equivalently, $D = \frac{p}{q}V$. 
Let $g$ be a defining equation for $V$ on $U$. 
The restriction of $\A$ to $U$ is the $B$-algebra
\[
A=\bigoplus_{i\in\ZZ}Bg^{-\floor{pi/q}}T^i\subset B_g[T,T^{-1}].
\]
We aim for an explicit description of the graded algebra $\Der_kA$. 
Note that $A_g=B_g[T,T^{-1}]$ and hence 
\begin{equation}\label{21}
\Der_kA\subset\Der_kA_g=A_g\Der_kB_g\oplus A_gT\p_T.
\end{equation}
Accordingly any derivation $\delta\in\Der_kA$ decomposes uniquely as 
\[
\delta=\sigma+\alpha T\p_T,\quad\sigma\in A_g\Der_kB_g,\quad\alpha\in A_g.
\]
For a homogeneous $\delta\in(\Der_kA)_d$, by degree reasons, 
\[
\sigma':=T^{-d}\sigma\in\Der_kB_g,\quad \alpha':=T^{-d}\alpha\in B_g.
\]
Correspondingly we consider
\begin{equation}\label{15}
\delta':=T^{-d}\delta=\sigma'+\alpha'T\p_T\in\Der_kA_g.
\end{equation}
For notational convenience, we abbreviate
\begin{equation}\label{79}
s_i:=m_{i+d}-m_i,\quad m_i:=\floor{pi/q}.
\end{equation}
Then
\[
\delta'=\sigma'+\alpha'T\p_T:Bg^{-m_i}T^i\to Bg^{-m_{i+d}}T^i
\]
and hence, cancelling $T^i$,
\[
\sigma'+i\alpha':Bg^{-m_i}\to Bg^{-m_{i+d}}.
\]
For any $b\in B$, this means that
\begin{equation}\label{14}
\sigma'(b)+b(i\alpha'-m_i\sigma'(g)/g)\in g^{-s_i}B.
\end{equation}
The minimal $s_i$ will be relevant for further investigations and we denote it by
\begin{equation}\label{80}
s:=\min\{s_i\mid i\in\ZZ\},
\end{equation}
Setting $b=1$, we see that \eqref{14} is equivalent to the following two conditions
\begin{align}
&\sigma':B\to g^{-s}B,\label{45}\\
&i\alpha'-m_i\sigma'(g)/g\in g^{-s_i}B \quad \text{for all $i\in\ZZ$}.\label{46}
\end{align}

Condition \eqref{45} implies that $\ord_g(g^s\sigma'(g)/g) \geq -1$. 
Two cases are possible:

\begin{enumerate}[(a)]
\item $\ord_g(g^s\sigma'(g)/g) \geq 0$ \label{54}
\item $\ord_g(g^s\sigma'(g)/g) = -1$ \label{55}
\end{enumerate}

We shall find out a numerical criterion, in terms of $p$, $q$ and 
\[
w:=\ch(k), 
\]
to tell when \eqref{55} and \eqref{46} can simultaneously hold. 
Our guiding principle is the following: 
If \eqref{55} holds then for $s_i=s$ the excessive pole of $-m_ig^s\sigma'(g)/g$ in \eqref{46} must be canceled by $ig^s\alpha'$.
Indeed, notice that \eqref{46} requires in particular that
\begin{equation}\label{56}
\ord_g(ig^s\alpha'-m_ig^s\sigma'(g)/g)\geq0\quad\text{if }s_i=s.
\end{equation}

We first need to see for what $i$ can $s_i$ achieve the minimum $s$. 
We use the notation $\{x\}:=x-\floor{x}$ in the following paragraphs.
Then
\begin{align*}
s_i&=\floor{p(i+d)/q}-\floor{pi/q}\\
&=\floor{pd/q}+\floor{pi/q} + \floor{\{pd/q\} + \{pi/q\}} -\floor{pi/q} \\
&= \floor{pd/q} + \floor{\{pd/q\} + \{pi/q\}}.
\end{align*}

Thus, we obtain the part~\eqref{36a} of the following lemma.
For part~\eqref{36b}, pick $i=q$ if $w\nmid q$ and, otherwise, pick $i\in\ZZ$ with $pi\equiv 1\mod q$.

\begin{lem}\label{36}
We have $s=\floor{pd/q}$.
Moreover, $s_i = s$ if and only if $\{pd/q\} + \{pi/q\} < 1$.
In particular,
\begin{enumerate}
\item\label{36a} if $pd\equiv -1\mod q$ then $s_i=s$ if and only if $i\in q\ZZ$;
\item\label{36b} if $pd\not\equiv -1\mod q$ or $w\nmid q$ then $s_i=s$ and $w\nmid i$ for some $i\in\ZZ$.\qed
\end{enumerate}
\end{lem}

We can now characterize the case where \eqref{55} and \eqref{46} both hold explicitly as follows.

\begin{lem}\label{32}
Assume that both \eqref{55} and \eqref{46} hold.
Then $pd \equiv -1 \mod q$ and $w\nmid q$. 
\end{lem}

\begin{proof}
We treat both the case $w\ne0$ and $w=0$ simultaneously.
Note that, in the latter case, $w\nmid i$ is equivalent to $i\ne0$ for $i\in\ZZ$.

Let $i=q$, then $m_i=p$, and $s_i=s$ by Lemma~\ref{36}. 
Then condition \eqref{56} reads
\begin{equation}\label{26}
\ord_g(g^s(q\alpha'-p\sigma'(g)/g)) \geq 0
\end{equation}

This inequality tells us that $w \nmid q$. 
In fact, if $w \mid q$, we would have $\ord_g(-pg^s\sigma'(g)/g) \geq 0$. 
By \eqref{55}, this would force $w \mid p$. 
However, $p$ and $q$ should have no common factor according to our initial choice.

Suppose $\{pd/q\}\neq(q-1)/q$. 
Take an integer $i$ such that $pi \equiv 1 \mod q$.
Such an integer exists because $p$ and $q$ are coprime. 
According to Lemma~\ref{36}, $s_{i+qj} = s$ for all $j\in\ZZ$.
Condition \eqref{56} with $i$ replaced by $i+qj$ reads
\begin{equation*}
\ord_g(g^s((i+qj)\alpha'-m_{i+qj}\sigma'(g)/g)) \geq 0.
\end{equation*} 
Comparing with \eqref{26} implies that
\begin{equation}\label{16}
\frac{m_{i+qj}}{i+qj} = \frac{p}{q}\quad\text{ in }k\text{ for all }j\in\ZZ\text{ such that }w \nmid (i+qj).
\end{equation}

Since $w \nmid q$, there exists a $j$ such that $w \nmid (i+qj)$. 
Thus the content of condition~\eqref{16} is non vacant. 

Using the definition $m_{i+qj}=\floor{p(i+qj)/q}$ and simplifying the above equation, we get
\begin{equation*}
pi/q-\floor{pi/q} = 0\quad\text{in $F_w\subset k$}.
\end{equation*}

This equation tells us that $pi-q\floor{pi/q}$ is a multiple of $w$ in $\ZZ$. 
However $pi-q\floor{\frac{p}{q}i} = 1$ by our choice of $i$. 
This is a contradiction. 
\end{proof}

There is also an inverse construction. 
Suppose $pd \equiv -1 \mod q$ and $w \nmid q$. 
Given a $\sigma' \in \Der_kB_g$ with the additional properties \eqref{45} and \eqref{55}, we can choose any
\begin{equation}\label{48}
\alpha':=\frac{p}{q}\frac{\sigma'(g)}{g}+g^{-s}a,\quad a\in B,
\end{equation}
and lift $\sigma'$ to $\delta' = \sigma' + \alpha'T\partial_{T}$. 
The condition \eqref{46} is satisfied because when $i \in q\ZZ$ the left hand side of \eqref{46} is $g^{-s} a$ by our choice of $\alpha'$.
If on the other hand $i \notin q\ZZ$ we have $s_i > s$ according to Lemma~\ref{36}; consequently, using \eqref{55},
\begin{equation*}
\ord_g(g^{s_i}(i\alpha'-m_i\sigma'(g)/g)) \geq \min\{\ord_g(g^{s_i}\sigma'(g)/g), \ord_g(g^{s_i}\alpha')\} \geq 0.
\end{equation*} 

It is time to explain the geometric meaning of this construction. 
Condition \eqref{45} tells us that restricting $\pi_*\Psi$ from \eqref{53} to $(\pi_*\Theta_C)_d\vert_U$ induces a morphism 
\begin{equation}\label{75}
\psi_{U}:(\pi_*\Theta_C)_d\vert_U\to\Theta_U(sV),\quad\delta\mapsto\sigma'=\pi_*\Psi(T^{-d}\delta).
\end{equation}
Lemma~\ref{32} combined with the above inverse construction shows that $\psi_U$ is surjective if and only if $pd \equiv -1 \mod q$ and $w \nmid q$.
If this numerical condition is not satisfied, then condition~\eqref{54} must hold and we see that the image of $\psi_{U}$ is contained in $\SDer_U(-\log V)(sV)$. 
The morphism is in fact surjective since for a given $\sigma'$ satisfying \eqref{54} we can lift it to a section of $(\pi_*\Theta_C)_d$ on $U$ by choosing any
\begin{equation}\label{49}
\alpha'\in g^{-s}B
\end{equation}
and $\delta'$ according to \eqref{15}. 
This lift apparently satisfies condition~\eqref{46}. 

Summarizing the above, we have

\begin{prp}\label{22}
Let $s=\floor{pd/q}$ and $w=\ch(k)$.
Then there is a short exact sequence based on the following numerical conditions.
If $pd\equiv-1\mod q$ and $w\nmid q$, then
\begin{equation}\label{86}
\xymat{
0\ar[r] & \O_U(sV)\ar[r]^-{\varphi_U} & (\pi_*\Theta_C)_d\vert_U\ar[r]^-{\psi_U} & \Theta_U(sV) \ar[r] & 0.
}
\end{equation}
If $pd\equiv-1\mod q$ and $w\mid q$, then
\[
\xymat{
0\ar[r] & \O_U(sV+V)\ar[r]^-{\varphi_U} & (\pi_*\Theta_C)_d\vert_U\ar[r]^-{\psi_U} & \SDer_U(-\log V)(sV) \ar[r] & 0.
}
\]
Otherwise,
\[
\xymat{
0\ar[r] & \O_U(sV)\ar[r]^-{\varphi_U} & (\pi_*\Theta_C)_d\vert_U\ar[r]^-{\psi_U} & \SDer_U(-\log V)(sV) \ar[r] & 0.
}
\]
\end{prp}

\begin{proof}
It remains to determine the kernel of $\psi_U$.
To this end, let $\delta\in\ker\psi_U$.
By \eqref{75}, this means that $\sigma'=0$.
By condition \eqref{46}, it follows that
\begin{equation}\label{77}
i\alpha'\in g^{-s_i}B,\quad\text{for all }i\in\ZZ.
\end{equation}
If $pd\equiv-1\mod q$ and $w\mid q$, then \eqref{77} implies $\alpha'\in g^{-s-1}B$ by Lemma~\eqref{36}.\eqref{36a}.
In all other cases, \eqref{77} implies  $\alpha'\in g^{-s}B$ by Lemma~\eqref{36}.\eqref{36b}.
This shows that 
\begin{equation}\label{76}
\varphi_U:\alpha'\mapsto\alpha T\p_T=\pi_*\Phi(T^d\alpha')
\end{equation}
induced by $\pi_*\Phi$ is the kernel of $\psi_U$.
\end{proof}

\begin{rmk}\label{84}
In the case where $D$ is a Cartier divisor and $w=\ch(k)=0$, we obtain an exact sequence
\begin{equation*}
\xymat{
0\ar[r] & \O_U(sD)\ar[r]^-{\varphi_U} & (\pi_*\Theta_C)_d\vert_U\ar[r]^-{\psi_U} & \Theta_U(sD) \ar[r] & 0.
}
\end{equation*}
similar to the sequence \eqref{86}, without having to assume each irreducible component of $D$ is Cartier.
To see this, let $g\in Q(B)$ be a defining equation for $D$ and read the preceding arguments in this section with $V=D$ and $p=q=1$. 
Note that, in this case, $m_i=i$ and $s_i=s$ for all $i\in\ZZ$ and hence \eqref{46} reduces to a single condition giving the lift as in \eqref{48}.
\end{rmk}

Finally, we deal with the general case where $D$ has multiple components as in \eqref{88}.

\begin{dfn}\label{12}
For each $d\in\ZZ$ and $w=\ch(k)$, define Weil divisors
\[
W_{D,d}:=\sum_{\substack{p_Vd\equiv-1\mod q_V\\\text{and }w\mid q_V}}V,\quad 
L_{D,d}:=\sum_{\substack{p_Vd\not\equiv-1\mod q_V\\\text{or }w\mid q_V}}V.
\]
\end{dfn}

Note $W_{D,d}=0$ if $w=0$ and that $L_{D,d}$ is reduced effective as required in Definition~\ref{29}.

\begin{rmk}\label{83}\
\begin{asparaenum}

\item\label{83a} In part \eqref{11a}$\leadsto$\eqref{11b} of Demazure's construction (see \S\ref{13}), one can choose different homogeneous rational $T\in Q(A)$, resulting in different $\QQ$-ample $\QQ$-Cartier divisors $D$.
However, the divisors $W_{D,d}$ and $L_{D,d}$ in Definition~\ref{12} are independent of this choice.
To see that one uses the fact that $\divisor(f)$ in \eqref{65} has integer coefficients which do not contribute to $W_{D,d}$ and $L_{D,d}$. 

\item\label{83b} By definition,
\[
\O_X(dD-L_{D,d})=\O_X(dD-D'),\quad D':=\sum_V(1-1/q_V)V.
\]
The divisor $D'$ occurs in Watanabe's formula relating the canonical module $K_A$ of $A$ to the canonical divisor $K_X$ of $X$,
\[
K_A=\bigoplus_{i\in\ZZ}H^0(X,\O_X(K_X+D'+iD)).
\]
In case $X$ is a regular curve, we have 
\[
\SDer_X(dD-\log L_{D,d})\cong\Theta_X(dD-L_{D,d})\cong\O_X(-K_X-D'+dD)
\]
and a generalized (short exact) Euler sequence was described already by Wahl (see \cite[Prop.~3.1, Proof of Thm.~3.3]{Wah15}.
However, the first isomorphism above is a special feature of this case.

\end{asparaenum}
\end{rmk}

\begin{prp}\label{47}
There is a short exact sequence
\[
\xymat{
0\ar[r] & \O_U(dD+W_{D,d})\ar[r]^-{\varphi_U} & (\pi_*\Theta_C)_d\vert_U\ar[r]^-{\psi_U} & \SDer_U(dD-\log L_{D,d}) \ar[r] & 0.
}
\]
\end{prp}

\begin{rmk}\label{41}
In particular, in case $D=\frac{p}{q}V$, $\O_U(dD+W_{D,d})=\O_U(sV+W_{D,d})$ which reduces to $\O_U(sV)$ or $\O_U(sV+V)$ depending on the numerical conditions in Proposition~\ref{22}.
Similarly, $\O_U(dD-L_{D,d})=\O_U(sV-L_{D,d})$ and hence
\[
\SDer_U(dD-\log L_{D,d})=\SDer_U(sV-\log L_{D,d})=\SDer_U(-\log L_{D,d})(sV)
\]
by Definition~\ref{29} and Lemma~\ref{34}.
Again this reduces to either $\Theta_U(sV)$ or $\SDer_U(-\log V)(sV)$ depending on the numerical conditions in Proposition~\ref{22}.
For this reason the sequence in Proposition~\ref{47} unifies the three sequences in Proposition~\ref{22}.
\end{rmk}

\begin{proof}[Proof of Proposition~\ref{47}]
We need to first establish the existence of the morphisms $\varphi_U$ and $\psi_U$. 
We denote the irreducible components and coefficients of $D$ by
\[
D=:\sum_i\frac{p_i}{q_i}V_i.
\]
There exists an affine open covering of $U'=U\setminus\bigcup_{i\ne j}V_i\cap V_j$ by $U_i:=U\setminus\bigcup_{i\neq j}V_j$. 
On each $U_i$, $D$ has only one irreducible component. 
Therefore, according to Remark~\ref{41}, we have an exact sequence on each $U_i$,
\[
\xymat{
0\ar[r] & \O_{U_i}(dD+W_{D,d})\ar[r]^-{\varphi_{U_i}} & (\pi_{*}\Theta_{C})_d\vert_{U_i}\ar[r]^-{\psi_{U_i}} & \SDer_U(dD-\log L_{D,d})\vert_{U_i} \ar[r] & 0.
}
\]
These exact sequences glue to an exact sequence on $U'$ because all $\varphi_{U_i}$ and $\psi_{U_i}$ are induced by $\Phi$ and $\Psi$ from \eqref{53} via \eqref{75} and \eqref{76}. 
By Lemmas~\ref{30}, \ref{24}, \ref{34}, pushing forward along $i:U'\to U$ yields a left exact sequence
\[
\xymat{
0\ar[r] & \O_{U}(dD+W_{D,d})\ar[r]^-{i_*\varphi_{U'}} & (\pi_{*}\Theta_C)_d\vert_{U}\ar[r]^-{i_*\psi_{U'}} & \SDer_U(dD-\log L_{D,d})\vert_{U}.
}
\]
We define $\varphi_U$ and $\psi_U$ by $i_*\varphi_{U'}$ and $i_*\psi_{U'}$, respectively.

To check surjectivity of the map $\psi_U$, it suffices to construct a lift along $\psi_U$ for any 
\[
\sigma\in H^0(U,\SDer_U(dD-\log L_{D,d})).
\]
Let $g_i$ be the defining equation for $V_i$ and denote by
\[
s^{(i)} := \min\{\floor{\frac{(j+d)p_i}{q_i}} - \floor{\frac{dp_i}{q_i}} \mid j\in\ZZ\} = \floor{\frac{p_id}{q_i}}
\]
the ``$s$'' for the basic case $D\vert_{U_i}=\frac{p_i}{q_i}V_i\cap U_i$ (see \eqref{79} and \eqref{80}).
Restricting to $U_i$, we find
\begin{equation}\label{81}
\ord_{g_i}(\sigma(b))\geq-s^{(i)}
\end{equation}
all $b\in B$ and all $i$ (see \eqref{54} and \eqref{55}). 
We consider $\delta:=\sigma+\alpha T\p_T\in\Der_k(K(C))$ where 
\begin{equation}\label{92}
\alpha:=\sum_i\alpha_i,\quad
\alpha_i:=
\begin{cases}
\frac{p_i}{q_i}\frac{\sigma(g_i)}{g_i},&\text{ if }p_id\equiv-1\mod q_i\text{ and }w\nmid q_i,\\
0,&\text{ otherwise}.
\end{cases}
\end{equation}
Regrouping the terms of $\alpha'=T^{-d}\alpha$ as
\[
\alpha'=\alpha'_i+\sum_{i\ne j}\alpha'_j
\]
and using \eqref{81}, we see that $\alpha'\vert_{U_i}$ has the shape required in \eqref{48} and \eqref{49} and hence $\delta \in H^0(U_i,(\pi_*\Theta_C)_d)$ by (the proof of) Proposition~\ref{22}.
Thus, $\delta\in H^0(U',(\pi_*\Theta_C)_d)= H^0(U,(\pi_*\Theta_C)_d)$ by gluing and Lemma~\ref{24}. 
By construction and \eqref{75}, $\delta$ lifts the given $\sigma$ proving surjectivity of $\psi_U$.
\end{proof}

\begin{mathias}

\begin{proof}[Mathias' proof of Proposition~\ref{47}]
By Remark~\ref{41}, the sequence exists on $U\setminus\bigcup_{i\ne j}D_i\cap D_j$.
By reflexivity of all members of the sequence and Lemma~\ref{30}, it suffices to check surjectivity of the map $\psi_U$.
To this end, for given $\sigma\in H^0(U,\SDer_U(dD-\log L_{D,d}))$, we consider the lift
\[
\delta:=\sigma+\alpha'T^dT\p_T\in\Der_k(K(C))
\]
of $\sigma$ where, denoting $V_i=:\divisor(g_i)$,
\[
\alpha':=\sum_j\alpha'_j,\quad
\alpha'_j:=
\begin{cases}
\frac{p_j\sigma'(g_j)}{q_jg_j},&\text{ if }p_jd\not\equiv-1\mod q_j\text{ and }w\nmid q_j,\\
0,&\text{ otherwise}.
\end{cases}
\]
Since $\alpha'$ specializes, locally for each $V$, to a valid choice of \eqref{48} and \eqref{49}, we have
$\delta\in H^0(U,(\pi_*\Theta_C)_d)$ by Lemmas~\ref{30} and \ref{24}.
\end{proof}

\end{mathias}

\begin{xia}

\begin{proof}[Xia's proof of Proposition~\ref{47}]
In the general situation, $D$ has multiple components. 
The analysis of the derivations on the cylinder is very similar to the basic case where $D$ has only one component. 
We will briefly demonstrate how it works out. 
We assume for simplicity that 
\begin{equation*}
D= \frac{p_1}{q_1}V_1 + \frac{p_2}{q_2}V_2,
\end{equation*}
and $V_1$ and $V_2$ are defined by $g$ and $h$ respectively. 
The same method which we shall use applies with no difficulty to the general case.

As we have done in the basic case, we use the notation:
\begin{align*}
m_i:=\floor{p_1i/q_1},&\quad n_i:=\floor{p_2i/q_2}\\
s_i:=m_{i+d}-m_i,&\quad t_i:=n_{i+d}-n_i\\
s:=\min\{s_i\mid i\in\ZZ\},&\quad t:=\min\{t_i\mid i\in\ZZ\}
\end{align*} 

The algebra $A$ now takes the form:
\[
A=\bigoplus_{i\in\ZZ}Bg^{-m_i}h^{-n_i}T^i\subset B_{gh}[T,T^{-1}].
\]

For a $\delta \in \Der_kA$, we can decompose it as
\[
\delta=\sigma+\alpha T\p_T,\quad\sigma\in A_{gh}\Der_kB_{gh},\quad\alpha\in A_{gh}.
\] 

If $\delta \in (\Der_kA)_d$, we let $\delta' = T^{-d}\delta, \sigma' = T^{-d}\sigma, \alpha'= T^{-d}\alpha$. 
They satisfy the conditions 
\[
\delta'+i\alpha':Bg^{-m_i}h^{-n_i}\to Bg^{-m_{i+d}}h^{-n_{i+d}} \quad \text{for all } i\in\ZZ.
\]

This means that
\[
\sigma'(b)+b(i\alpha'-m_i\frac{\sigma'(g)}{g}-n_i\frac{\sigma'(h)}{h})\in g^{-s_i}h^{-t_i}B \quad\text{ for all }b\in B.
\]

Or equivalently
\begin{align}
&\sigma':B\to g^{-s_i}h^{-t_i}B \quad \text{for all }i.\label{58}\\
&i\alpha'-m_i\frac{\sigma'(g)}{g} - n_i\frac{\sigma'(h)}{h}\in g^{-s_i}h^{-t_i}B \quad \text{for all }i.\label{59}
\end{align}

By Lemma~\ref{36}, $s_i$ and $t_i$ obtain the minimums simultaneously when $i=q_1 q_2$. 
Therefore \eqref{58} is equivalent  to 
\begin{equation}\label{60}
\sigma':B\to g^{-s}h^{-t}B.
\end{equation}

This tells us that there is always a map 
\begin{equation*}
(\pi_*(\Theta_C))_d\vert_U \to \Theta_U(sV_1 + tV_2)
\end{equation*}

Condition $\eqref{58}$ implies that 
\begin{equation*}
\frac{\sigma'(h)g^{s_i}}{h} \in B_h \quad\text{for all }i\in\ZZ.
\end{equation*}

Therefore condition \eqref{59} implies that 
\begin{equation*}
i\alpha'g^{s_i}-m_i\frac{g^{s_i}\sigma'(g)}{g} \in B_h \quad\text{for all }i\in\ZZ.
\end{equation*}

On the other hand, condition \eqref{60} implies that 
\begin{equation*}
\sigma' : B_h \to g^{-s}B_h
\end{equation*}

The last two conditions reduces us to the situation of \eqref{45} and \eqref{46}. 
Thus, if $p_1d \not\equiv d-1\ (\textrm{mod}\ q_1)$ or $w \mid q$, we have $g^s\sigma'(g)\in gB_h$ (because condition \eqref{55} and \eqref{46} implies the numerical condition for the triple $(p_1,q_1,d)$, and if the numerical condition fails, we must have condition \eqref{54}). 
We also know that $g^sh^t\sigma'(g) \in B$ according to \eqref{60}. 
Therefore we get $g^sh^t\sigma'(g)\in B\cap gB_h = gB$. 
A similar analysis can be done for $g^sh^t\sigma'(h)$.

This tells us that there is a morphism
\begin{equation*}
(\pi_*(\Theta_C))_d\vert_U \to \SDer_U(-\log V_1)(sV_1 + tV_2)
\end{equation*}
if the numerical condition for the triple $(p_1,q_1,d)$ is not satisfied.

In the general case, we must have a morphism
\begin{equation*}
(\pi_*(\Theta_C))_d\vert_U \to \SDer_U(-\log L_d)(\sum_i s_iV_i)
\end{equation*}
where $L_d$ is the union of the $V_i$ whose coefficients do not satisfy the corresponding numerical condition.

Given a $\sigma' \in H^0(U,\SDer_U(-\log L_d)(\sum_i s_iV_i))$, we set $\alpha' = \sum_i\alpha'_i$ where $\alpha'_i = \frac{p_i\sigma'(g_i)}{q_ig_i}$ if the numerical condition for $(p_i,q_i,d)$ is satisfied, and $\alpha'_i = 0$ if the numerical condition is not satisfied. 
We claim $\delta' = \sigma' + \alpha'T\partial_t$ is a lift of $\sigma'$ to $H^0(U,(\pi_*(\Theta_C))_d)$. 
One must check $\alpha'$ and $\sigma'$ satisfy a condition analogous to \eqref{59}. 
Let me show how this can be done if $D$ has only 2 irreducible components. 
In this case, we are bound to show that
\begin{equation*}
i(\alpha'_1 + \alpha'_2)g^{s_i}h^{t_i}- m_i\frac{g^{s_i}h^{t_i}\sigma'(g)}{g}-n_i\frac{g^{s_i}h^{t_i}\sigma'(h)}{h} \in B \quad\text{for all $i$.}
\end{equation*}
Or equivalently, to show
\begin{equation*}
(i\alpha'_1 - m_i\frac{\sigma'(g)}{g})g^{s_i}h^{t_i} + (i\alpha'_2 - n_i\frac{\sigma'(h)}{h})g^{s_i}h^{t_i} \in B \quad\text{for all }i.
\end{equation*}

But for any $i$, each summand is already in $B$ by our choice of $\alpha'$. 
\end{proof}

\end{xia}

\begin{rmk}\label{57}
In the special case $q=1$, it is clear from our construction that 
\[
(\pi_*\Theta_C)_{d_2}\vert_U=(\pi_*\Theta_C)_{d_1}\vert_U\otimes_{\O_U}\O_U((d_2-d_1)D).
\]
\end{rmk}

\section{Generalized Euler sequence}

We can now globalize our local construction from \S\ref{71}. 

\begin{prp}\label{23}
Denote by $Z:=X\setminus X^\reg$ the complement of the regular part of $X$.
Then, for each $d\in\ZZ$, there is an exact sequence
\begin{equation}\label{25}
\xymat@C=8pt{
0\ar[r]&\O_X(dD+W_{D,d})\ar[r]&(\pi_*\Theta_C)_d\ar[r]&\SDer_X(dD-\log L_{D,d})\ar[r]&\SH_Z^2(\O_X(dD+W_{D,d}))
}
\end{equation}
whose first three terms are coherent reflexive $\O_X$-modules.
\end{prp}

\begin{proof}
Denote by $i:X^\reg\into X$ the inclusion.
By Proposition~\ref{47}, 
\begin{equation}\label{43}
0\to i^*\O_X(dD+W_{D,d})\to i^*(\pi_*\Theta_C)_d\to i^*\SDer_X(dD-\log L_{D,d})\to0
\end{equation}
is exact.
Applying $i_*$ yields
\begin{gather}\label{40}
0\to i_*i^*\O_X(dD+W_{D,d})\to i_*i^*(\pi_*\Theta_C)_d\to i_*i^*\SDer_X(dD-\log L_{D,d})\to\cdots\\
\nonumber\cdots\to R^1i_*i^*\O_X(dD+W_{D,d}).
\end{gather}
By Lemmas~\ref{24} and \ref{34}, $(\pi_*\Theta_C)_d$ and $\SDer_X(dD-\log L_{D,d})$ are coherent reflexive $\O_X$-modules.
By Lemma~\ref{30}, $\F=i_*i^*\F$ for any coherent reflexive $\O_X$-module and, by \cite[Cor.~1.9]{Har67}, $R^1i_*i^*=\SH_Z^2$.
This turns \eqref{40} into \eqref{25}.
\end{proof}

Motivated by Lemma~\ref{39}, we would like to derive a short exact sequence of global sections from \eqref{25}.
This is possible in case $\O_X(dD+W_{D,d})$ has sufficient depth.
Restricting ourselves to the case where $W_{D,d}=0$, we are concerned with the depth of $\O_X(dD)\cong\O_X(d)$.
Watanabe~\cite[Rem.~(2.11)]{Wat81} showed that $X$ is Cohen--Macaulay if $A$ is so.
More generally, Flenner~\cite[(2.3)~Kor.]{Fle81b} showed that Serre's condition $(S_\ell)$ passes from $\Gamma$ to $X$.
We extend his result to $\O_X(i)$ as follows.

\begin{draft}

\begin{prp}\label{42}
Assume that $A$ is a Cohen--Macaulay ring.
Then, for each $d\in\ZZ$, $\O_X(d)$ is a Cohen--Macaulay $\O_X$-module.
\end{prp}

\begin{proof}
For the Veronese subring $A^{(N)}\subset A$, we have $X=\Proj A=\Proj A^{(N)}$.
The $\O_X$-module $\M:=\O_X(d)=\tilde M$ is associated to the graded $A^{(N)}$-module
\[
M:=A^{(N)+d}:=\bigoplus_{i\in\NN}A_{iN+d}.
\]
Note that $A$ and $A^{(N)}$ are *local with *maximal ideals $A_+$ and $A^{(N)}_+$ respectively.
By \cite[Prop.~1.5.15]{BH93},
\[
\grade(A_+,A)=\depth A_{A_+}=\dim A_{A_+}=\dim A=:n.
\]
By \cite[Prop.~1.5.11]{BH93}, there is an $A$-regular sequence $x_1,\dots,x_n\in A_+$ of homogeneous elements.
Replacing each $x_i$ by $x_i^N$, we may assume that $x_1,\dots,x_n\in A^{(N)}_+$ form a regular sequence on $M$.
By \cite[Prop.~1.5.15]{BH93} and since $A^{(N)}\subseteq A$ is integral, we have
\[
n\le\grade(A^{(N)},M)=\depth M_{A^{(N)}_+}\le\dim M_{A^{(N)}_+}\le\dim A^{(N)}_{A^{(N)}_+}=\dim A_{A_+}=n.
\]
By \cite[Exc.~2.1.27]{BH93}, $M$ is then a maximal Cohen--Macaulay module.

We are therefore reduced to showing that for an arbitrary Cohen--Macaulay $A$-module $M$ the associated coherent $\O_X$-module $\M:=\tilde M$ is Cohen--Macaulay assuming that $A$ satisfies the condition (\#) of \cite[Ch.~5, \S1]{GW78}.
Let $\pp\lhd A$ be a graded prime ideal.
Then $M_{(\pp)}$ is a Cohen--Macaulay module over the Cohen--Macaulay ring $A_{(\pp)}$ by \cite[Exc.~2.1.27]{BH93}.
The argument in the proof of \cite[Lem.~(5.1.10)]{GW78} completes the proof.
\end{proof}

\end{draft}

\begin{xia}

\begin{proof}[Xia's proof of Proposition~\ref{42}]
Let $\pp$ be a homogeneous prime of $A$, corresponding to a closed point of $\Proj A$. 
We have $\dim A_{\pp} = \grade(\pp A_{\pp}, A_{\pp}) = \grade(\pp,A)$. 
Both equalities need the Cohen--Macaulayness of $A$. 
For the second equality, recall that 
\begin{equation*}
\grade(\pp,A) = \min\{i,\Ext_A^{i}(A/\pp,A) \neq 0\},
\end{equation*}
and $\Ext_A^{i}(A/\pp,A) = 0$ if and only if $\Ext_{A_{\mm}}^{i}(A_{\mm}/\pp A_{\mm},A_{\mm}) = \Ext_A^{i}(A/\pp,A) \otimes_A A_{\mm} = 0$ for all maximal ideal $\mm$ lying above $\pp$. 
Therefore we can find a maximal ideal $\mm$ of $A$, lying above $\pp$, such that $\grade(\pp,A) = \grade(\pp A_{\mm}, A_{\mm})$. 
The latter is equal to $\grade(\pp A_{\pp}, A_{\pp})$, a well known property whose proof can be found in Matsumura's commutative ring theory theorem 17.3.

If $\dim A_{\pp} = n$, then we can find a homogeneous regular $A$-sequence $x_1, \ldots, x_n$, which is contained in $\pp$. 
For any positive integer $N$, $x_1^N, \ldots, x_n^N$ is a homogeneous regular sequence for both $A$ and $A^{(N)}$, contained in $\pp \cap A^{(N)}$.   

Now, Let $N$ be an integer such that $ND$ is an ample Cartier divisor on $X = \Proj A$. 
Let $R = A^{(N)}$ be the $N$-th Veronese subring of $A$, let $\mathfrak{q} = \pp \cap R$, also let 
\begin{equation*}
M = \bigoplus_{i \geq 0}A_{iN+d}. 
\end{equation*}  

We observe that $M$ is a graded $R$-module, and the coherent sheaf $\tilde{M}$ on $\Proj R$ is isomorphic to $\O_X(d)$ under the isomorphisms $\Proj R \cong \Proj A \cong X$. 
The construction above implies that we can find a homogeneous $M$-sequence of length $n$ in $\qq$. 

After passing from $A$ to its Veronese subring $R$, we see that $R$ satisfies the condition (\#) of \cite[Ch.~5, \S1]{GW78} and $\mathfrak{q}$ corresponds to a closed point on $\Proj R$. 
The analysis in \cite[5.1.10]{GW78} shows that $R_{(\mathfrak{q})} \cong (R_{(\mathfrak{q})})_0[T,T^{-1}]$ and $M_{(\mathfrak{q})} \cong (M_{(\mathfrak{q})})_0[T,T^{-1}]$. 
The previously constructed homogeneous $M$-sequence of length $n$ now gives rise to a homogeneous $M_{(\mathfrak{q})}$-sequence of length $n$, contained in $\mathfrak{q} R_{(\mathfrak{q})}$. 
By multiplying powers of $T$, we can shift all elements in this sequence to degree $0$ and get a new $M_{(\mathfrak{q})}$-sequence whose elements are in $(R_{(\mathfrak{q})})_0$. 
It is clear that this sequence is also a $(M_{(\mathfrak{q})})_0$-sequence. 

This shows that $\grade(\mathfrak{q}(R_{(\mathfrak{q})})_0, (M_{(\mathfrak{q})})_0) \geq n = \dim A_{\pp} = \dim (R_{(\mathfrak{q})})_0$. 
Since the depth of $(M_{(\mathfrak{q})})_0$ is always less or equal to $\dim (R_{(\mathfrak{q})})_0$, we conclude that $\O_X(d) \cong \tilde{M}$ is Cohen--Macaulay. 
\end{proof}

\end{xia}

\begin{lem}\label{31}
Let $A$ be a finitely generated positively graded algebra over a field $k$ of characteristic $\ch k=0$.
Let $f\in A_d$ be a non zero divisor.
If $A$ satisfies Serre's condition $(S_\ell)$ then the same holds true for the $(A_f)_0$-modules $(A_f)_i=(A(i)_f)_0$ where $i\in\ZZ$.
\end{lem}

\begin{proof}\pushQED{\qed}
Clearly $(A_f)_i\otimes_k\bar k\cong((A\otimes_k\bar k)_{f\otimes 1})_i$ and $A$ satisfies $(S_\ell)$ if and only if $A\otimes_k\bar k$ does.
We may therefore assume that $k=\bar k$.
Let $\ZZ_d$ be the cyclic group of order $d$.
For any $k[\ZZ_d]$-module $M$ we denote by $M^{\bar i}$ the eigenspace for the character $\bar 1\mapsto\zeta_d^i$ where $\zeta_d\in k$ is a primitive $d$-th root of unity.
Setting $A_i\subset A^{\bar i}$, $A$ becomes a $k[\ZZ_d]$-module.
Since $f\in A_d$ is $\ZZ_d$-invariant, there is an induced action on $B:=A/\ideal{f-1}$.
One easily shows (see \cite[Thm.~3.1]{Dol03}) that $B$, and hence each $B^{\bar i}$, is a finitely generated module over the finitely generated $k$-algebra $B^{\bar 0}$.
By \cite[\S2]{Fle81b}, $B$ is $(S_\ell)$ and $B^{\bar 0}=(A_f)_0$; analogously, $B^{\bar i}=(A_f)_i$.
For $\pp\in\Spec B^{\bar 0}$, $(B^{\bar i})_\pp\cong(B_\pp)^{\bar i}$.
We may therefore assume that $B^{\bar 0}$ is local with maximal ideal $\mm$.
The primes over $\mm$ are conjugates under $\ZZ_d$.
In fact, assume that there are two invariant non-empty sets $\{\pp_i\}$ and $\{\qq_j\}$ of primes over $\mm$.
Then prime avoidance yields an element $b\in\bigcap_i\pp_i\setminus\bigcup_j\qq_j$.
Its norm $b_0=\prod_{\bar i\in\ZZ_d}\bar i.b$ has the same property and $b_0\in\pp_i\cap B^{\bar 0}=\mm\subset\qq_j$, contradicting to $b_0\not\in\qq_j$.
It follows the $\dim B_\pp=\dim B^{\bar 0}$ for all $\mm\subset\pp\in\Spec B$.
By \cite[Prop.~1.2.10.(a)]{BH93} and since $B$ is $(S_\ell)$,
\[
\grade(\mm B,B)=\min\{\depth B_\pp\mid\mm\subset\pp\in\Spec B\}\ge\min\{\ell,\dim B^{\bar 0}\}
\]
Now we apply the argument from the proof of \cite[Prop.~1]{Fog81}:
If $b\in\mm B$ is $B$-regular, then its norm $b_0=\prod_{\bar i\in\ZZ_d}\bar i.b$ is $B$-regular and $b_0\in\mm B\cap B^{\bar 0}=\mm$.
Moreover, $(B/Bb_0)^{\bar 0}=B^{\bar 0}/B^{\bar 0}b_0$ using the Reynolds operator (which is defined due to the hypothesis $\ch k=0$).
By \cite[Prop.~1.2.10.(d)]{BH93} and induction, it follows that 
\[
\grade(\mm,B)=\grade(\mm B,B)\ge\min\{\ell,\dim B^{\bar 0}\}.
\]
Since $B\cong\bigoplus_{\bar i\in\ZZ_d}B^{\bar i}$ as $B^{\bar 0}$-module, any $B$-regular sequence in $\mm$ is also $B^{\bar i}$-regular.
Thus,
\[
\grade(\mm,B^{\bar i})\ge\grade(\mm,B)\ge\min\{\ell,\dim B^{\bar 0}\}\ge\min\{\ell,\dim B^{\bar i}\}.\qedhere
\]
\end{proof}

\begin{mathias}

\begin{proof}[Mathias' failed ambitious attempt to prove Lemma~\ref{31}]
We follow the lines of \cite[\S2]{Fle81b}.
There is an inclusion of graded $k$-algebras
\[
A\into A'[T^{\pm1}],\quad \sum_ia_i\mapsto\sum_ia_iT^i
\]
where the grading on the target is defined by $\deg T=1$ (forgetting the grading of $A$).
Pick $f\in A_d$, $d>0$, then the above map factors 
\begin{equation}\label{108}
\xymat{
A\ar@{^(->}[r]^-\varphi\ar[d]_\alpha & A'[T^{\pm1}]\ar[d]_\beta\\
A_f\ar[r]^-{\varphi_f} & A'[T^{\pm1}]/\ideal{f-1}\rlap{$=(A'/\ideal{f-1})[T^{\pm1}]$}
}
\end{equation}
More generally, for a finitely generated graded $A$-module $M$, there is an inclusion of graded modules over $\varphi$
\begin{equation}\label{110}
M\into M'[T^{\pm1}]=M'\otimes_{A'}A'[T^{\pm1}],\quad \sum_im_i\mapsto\sum_im_iT^i.
\end{equation}
Tensoring with $\varphi_f$, we obtain a diagram of graded modules over \eqref{108}
\begin{equation}\label{111}
\xymat{
M\ar@{^(->}[r]^-\psi\ar[d] & M'[T^{\pm1}]\ar[d]\\
\llap{$M_f=$}A_f\otimes_\alpha M\ar[r]^-{\varphi_f\otimes\psi} & (A'/\ideal{f-1})[T^{\pm1}]\otimes_\beta M'[T^{\pm1}]\rlap{$=(M'/\ideal{f-1}M')[T^{\pm1}]$.}
}
\end{equation}

Let $\mu_d=\ideal{\zeta}$ be the cyclic group of order $d$ generated by a primitive $d$-th root of unity.
It acts on $A_i$ and $M_i$ by multiplication by $\zeta^i$ and on $T$ by multiplication by $\zeta$.
As $f$ is $\mu_d$-invariant, there is an induced action on $A'/\ideal{f-1}$ and on $M'/\ideal{f-1}M'$.
By \cite[(2.1) Lem.]{Fle81b}, there is a commutative diagram
\begin{equation}\label{109}
\xymat{
A_f\ar[r]_-\cong & (A'/\ideal{f-1})[T^{\pm1}]^{\mu_d}\ar@{^(->}[r] & (A'/\ideal{f-1})[T^{\pm1}]\\
(A_f)_0\ar[r]_-\cong\ar@{^(->}[u] & (A'/\ideal{f-1})^{\mu_d}\ar@{^(->}[r]\ar@{^(->}[u] & A'/\ideal{f-1}\ar@{^(->}[u]
}
\end{equation}
where the lower row is the degree-$0$ part of the upper one and the middle column the induced map on the $\mu_d$-invariants of the right column.
The same argument with $A$ replaced my $M$ yields a commutative diagram over \eqref{109}
\[
\xymat{
M_f\ar[r]_-\cong & (M'/\ideal{f-1}M')[T^{\pm1}]^{\mu_d}\ar@{^(->}[r] & (M'/\ideal{f-1}M')[T^{\pm1}]\\
(M_f)_0\ar[r]_-\cong\ar@{^(->}[u] & (M'/\ideal{f-1}M')^{\mu_d}\ar@{^(->}[r]\ar@{^(->}[u] & M'/\ideal{f-1}M'\ar@{^(->}[u]
}
\]
In general, because of \eqref{110}, it is unclear whether  
\begin{equation}\label{112}
(M'/\ideal{f-1}M')[T^{\pm1}]=(A'/\ideal{f-1})[T^{\pm1}]\otimes_{\varphi_f}M_f.
\end{equation}
However, specializing to the case $M=A(e)$,
\[
M\otimes_{\varphi}A'[T^{\pm1}]=M'[T^{\pm1}](e).
\]
Thus, by \eqref{110} and \eqref{111}, \eqref{112} holds true after forgetting degrees:
\[
(M'/\ideal{f-1}M')[T^{\pm1}]'=(A'/\ideal{f-1})[T^{\pm1}]'\otimes_{\varphi_f'}M'_f.
\]
By hypothesis, the $A'_f$-module $M'_f$ satisfies $(S_\ell)$ and $\varphi_f'$ is \'etale by \cite[(2.2) Lem.]{Fle81b}.
It follows that the $A'/\ideal{f-1}$-module $M'/\ideal{f-1}M'$ satisfies $(S_\ell)$.
\end{proof}

\end{mathias}

\begin{prp}\label{19}
If $\ch k=0$ then global sections of \eqref{25} form a short exact sequence
\[
\xymat{
0\ar[r]&H^0(X,\O_X(d))\ar[r]&(\Der_kA)_d\ar[r]&H^0(X,\SDer_X(dD-\log L_{D,d}))\ar@{-->}[r]&0,
}
\]
with the dashed arrow if $A$ is $(S_3)$ and $X$ is $(R_2)$ of dimension $\dim X\ge2$.
\end{prp}

\begin{proof}
The first claim follows from Proposition~\ref{23} using Lemma~\ref{39}.
Now assume the additional hypotheses.
Since $\dim X\ge2$ and $A$ is $(S_3)$, we have $\depth A\ge 3$ and hence $H^1(X,\O_X(d))=0$ for all $d\in\ZZ$ by \cite[Cor.~2.3]{Wat81}.
Since $X$ is $(R_2)$, we have $\codim_XZ\ge3$.
We may therefore assume that $\dim X\ge3$ as otherwise $Z=\emptyset$.
By Lemma~\ref{31}, $\O_X(d)$ is $(S_3)$ and, in particular, $\depth_Z\O_X(d)\ge 3$.
By \cite[Thm.~3.8]{Har67}, this gives $\SH_Z^i(\O_X(d))=0$ and hence the second claim.
\end{proof}

For the remainder of the section, we assume that $X$ is regular, that $W_{D,d}=0$, and that $L_{D,d}$ is a free divisor.
In this case, \eqref{25} is a short exact sequence whose extension class can be described explicitly.
Due to the first two hypotheses, the sequence reads 
\begin{equation}\label{91}
\xymat{
0\ar[r]&\O_X(dD)\ar[r]^-\varphi&(\pi_*\Theta_C)_d\ar[r]^-\psi&\SDer_X(dD-\log L_{D,d})\ar[r]&0.
}
\end{equation}
By definition, the extension class of a short exact sequence of $\O_X$-modules 
\begin{equation}\label{100}
0\to\F\to\E\to\G\to0
\end{equation}
is the image of the identity morphism under the connecting homomorphism
\begin{equation}\label{97}
\xymat{
\Hom_{\O_X}(\F,\F)\ar[r]^\p&\Ext^1_{\O_X}(\G,\F);
}
\end{equation}
it parametrizes extensions \eqref{100} of $\F$ by $\G$ up to isomorphism.
As one derives from the spectral sequence associated to the composition of functors $\Hom_{\O_X}(-,-)=H^0\circ\SHom_{\O_X}(-,-)$, the target of \eqref{97} fits into an exact sequence
\begin{gather}\label{99}
0\to H^1(X,\SHom(\G,\F))\to\Ext^1_{\O_X}(\G,\F)\to H^0(X,\SExt^1_{\O_X}(\G,\F))\to\cdots\\
\nonumber\cdots\to H^2(X,\SHom(\G,\F))\to0.
\end{gather}
Due to the hypotheses $\SDer_X(-\log L_{D,d})$ is locally free and, by Lemma~\ref{34}.\eqref{34b} and Remark~\ref{101}, 
\begin{align*}
\SHom_{\O_X}(\SDer_X(dD-\log L_{D,d}),\O_X(dD))
&\cong\SHom_{\O_X}(\SDer_X(-\log L_{D,d})(dD),\O_X(dD))\\
&\cong\SHom_{\O_X}(\SDer_X(-\log L_{D,d}),\O_X)\\
&\cong\Omega_X^1(\log L_{D,d}).
\end{align*} 
It follows from \eqref{99} that the extension class of the sequence \eqref{91} is the image of $1\in H^0(X,\O_X)$ in $H^1(X,\Omega_X^1(\log L_{D,d}))$ under the connecting homomorphism in the long exact cohomology sequence associated to the dual sequence
\begin{equation}\label{96}
0\gets\SHom_{\O_X}(\O_X(dD),\O_X(dD))\gets\SHom_{\O_X}((\pi_*\Theta_C)_d,\O_X(dD))\gets\Omega^1(\log L_{D,d})\gets0.
\end{equation}
This image can be described explicitly as a \v{C}ech cocycle.
To this end, take an affine open covering
\[
X=\bigcup_\alpha X_\alpha
\]
and local defining equations $g^{(\alpha)}_i$ for $V_i$ on $X_\alpha$.
First, restrict the sequence \eqref{96} to each $X_\alpha$ and lift $\id_{\O_{X_\alpha}(dD)}$ to the middle term with the help of the splitting of the sequence \eqref{91} on $X_\alpha$ defined by \eqref{92}.
Rewriting \eqref{92}, this lift is given explicitly by
\[
\delta\mapsto\omega_\alpha(\psi(\delta)),\quad\omega_\alpha:=\sum_{p_id\equiv -1\mod q_i}\frac{p_i}{q_i}\frac{dg_i^{(\alpha)}}{g_i^{(\alpha)}}.
\]
Applying the \v{C}ech differential to the $0$-cocycle $(\omega_\alpha\circ\psi)_\alpha$, we obtain a $1$-cocycle $\varepsilon_D\in H^1(X,\Omega_X^1)$ given by
\begin{align}\label{93}
(\varepsilon_D)_{\alpha,\beta}
&:=\sum_{p_id\equiv -1\mod q_i}\frac{p_i}{q_i}\left(\frac{dg^{(\beta)}_i}{g^{(\beta)}_i}-\frac{dg^{(\alpha)}_i}{g^{(\alpha)}_i}\right)\\
\nonumber&\text{``}=d\log\prod_{p_id\equiv -1\mod q_i}\left(\frac{g^{(\beta)}_i}{g^{(\alpha)}_i}\right)^{\frac{p_i}{q_i}}\text{''}.
\end{align}

The preceding arguments now prove the following generalization of a result of Wahl (see \cite[Prop.~3.3]{Wah76}).

\begin{prp}\label{94}
Assume that $X$ is regular, that $W_{D,d}=0$, and that $L_{D,d}$ is a free divisor.
Then the extension class of the short exact sequence \eqref{91} is the image of $\varepsilon_D$ under the morphism induced by the natural inclusion
$\Omega_X^1\to\Omega_X^1(\log L_{D,d})$.\qed
\end{prp}

\section{Examples}\label{38}

\subsection{A normal ICIS with negative derivation}\label{38a}

We apply our result to a graded isolated complete intersection singularity with negative derivation found in \cite{GS14}.
Let $P=k[x_1,\dots,x_6]$ be the graded polynomial algebra with weights $w=(8,8,5,2,2,2)$ on the variables $x_1,\dots,x_6$.
The weighted homogeneous polynomials
\begin{align*}
g_1:=&x_1x_4+x_2x_5+x_3^2-x_4^5,\\
g_2:=&x_1x_5+x_2x_6+x_3^2+x_6^5
\end{align*}
of degree $10$ define a $4$-dimensional graded isolated complete intersection singularity $A:=P/\ideal{g_1,g_2}$ with a derivation 
\begin{equation}\label{18}
\eta:=
\begin{vmatrix}
\p_1 & \p_2 & \p_3 \\
x_4 & x_5 & 2x_3 \\
x_5 & x_6 & 2x_3 
\end{vmatrix}
=2x_3(x_5-x_6)\p_1-2x_3(x_4-x_5)\p_2+(x_4x_6-x_5^2)\p_3
\end{equation}
of degree $-1$.
Consider the corresponding $3$-dimensional projective $k$-variety $X:=\Proj A\subset\PP_k^w$.
By Delorme's reduction~\cite[1.3.~Prop.]{Del75} (see also \cite[1.3.1.~Prop.]{Dol82}), the second Veronese subring of $P$ is the polynomial subring $P^{(2)}=k[x_1,x_2,y,x_4,x_5,x_6]$ where $y=x_3^2$ is a new variable of weight $10$.
Because $g_1$ and $g_2$ are contained in $P^{(2)}$, the second Veronese subring of $A$ is
\[
A^{(2)}=P^{(2)}/\ideal{g_1,g_2}
\]
where 
\begin{align*}
g_1=&x_1x_4+x_2x_5+y-x_4^5,\\
g_2=&x_1x_5+x_2x_6+y+x_6^5.
\end{align*}
Since $y$ occurs as a linear term of $g_1$ and $g_2$, there is an isomorphism
\[
A^{(2)}=P^{(2)}/\ideal{g_1-g_2,g_1}\cong P^{(2)}/\ideal{g,y}=k[x_1,x_2,x_4,x_5,x_6]/\ideal{g}\cong P'/\ideal{g}=:A'
\]
where
\[
g:=g_1-g_2=x_1x_4-x_1x_5+x_2x_5-x_2x_6-x_4^5-x_6^5
\]
and $P':=k[x_1,x_2,x_4,x_5,x_6]$ is the polynomial ring with weights renormalized to $w'=4,4,1,1,1$.
Then we have
\[
X=\Proj A\cong\Proj A^{(2)}\cong\Proj A'.
\]

Thus, $A'$ is a graded isolated hypersurface singularity and we can consider $X\subset\PP_k^{w'}$.
In particular, 
\begin{equation}\label{44}
X^\sing\subseteq V(x_4,x_5,x_6).
\end{equation}
In order to pass to the chart $x_1=1$ we introduce a $4$th root of $x_1$ and divide $g$ by its $5$th power.
As a result we obtain
\[
g=(1-x_4^4)x_4-(1-x_2)x_5-(x_2+x_6^4)x_6.
\]
Then $X\vert_{x_1=1}$ is the quotient of the affine $k$-variety defined by $g$ by the cyclic group $\mu_4$ acting diagonally on $x_4,x_5,x_6$.
Unless $x_2=1$, $g$ serves to eliminate $x_5$ and hence, using the notation from \cite[\S2.4]{Kol07},
\[
X\vert_{x_1=1\ne x_2}\cong(\AA_k^1\setminus\{0\})\times(\AA_k^2/\frac14(1,1))
\]
corresponding to variables $x_2,x_4,x_6$.
An analogous argument applies to the chart $x_2=1$.
Therefore, \eqref{44} is an equality
\[
X^\sing=V(x_4,x_5,x_6).
\]
In particular, $\codim_XX^\sing=2$ and $X$ is not $(R_2)$.

Following Demazure's construction (see \cite[(2.6) Lem.~(iii), (2.10)]{Dem88}), one can choose $T=x_3/x_4^2$ and hence $D=\frac12 V(x_3)-2V(x_4)$.
In particular, $L_{D,-1}=0$ (see Definition~\ref{12}).
Since $H^0(X,\O_X(-1))=0$, Proposition~\ref{19} serves to interpret
\[
\eta\in H^0(X,\SDer_X(-D)).
\]

\subsection{A nonsingular plane conic}\label{38b}

In this section, we compute explicitly the sheaves $\M_{D,d}=(\pi_*\Theta_C)_d$ on the toy model $(X,D)$ where $X=\PP^1_k$ with $\O_X(D)=\O_X(2)$. 
The section ring is isomorphic to the subring
\[
A:=k[x^2,xy,y^2]\subset k[x,y], 
\]
where the generators $x^2,xy,y^2$ all have degree $1$. 
We see that $y^{-1}\partial_x$ is a derivation of degree $-1$ if $\ch(k)=2$. 
This can also be seen by passing to the isomorphic ring
\[
A\cong k[x,z,y]/\ideal{z^2-xy}
\]
where the degrees of $x,y,z$ are all $1$. 
There is no derivation of degree $-1$ if $\ch(k) \neq 2$, but $\partial_z$ is such a derivation if $\ch(k) = 2$. 

We follow closely the idea used in our construction of the generalized Euler sequence.
In Demazure's construction (see \S\ref{13}), one has to choose a homogeneous fraction of degree $1$ in $K(C)$, and uses it to define an ample $\QQ$-Cartier $\QQ$-divisor. 
We choose $T=xy$ for our purpose and pass to the isomorphic ring
\[
k(t,T)\supset k[tT,T,\frac Tt]\cong A,\quad T\mapsto xy,\quad t\mapsto\frac xy.
\]
In this particular model for the section ring $A$, $\deg t=0, \deg T=1$, and $t$ is the parameter for $\PP^1_k$. 
The ample $\QQ$-Cartier $\QQ$-divisor $D$ on $\PP^1_k$ is defined by $\divisor(T)=\pi^{-1}(D)$. 
It is clear that
\[
D=\{0\}+\{\infty\}
\]
because $xy=0$ implies either $x=0$ or $y=0$.

We use the standard affine cover of $\PP^1_k$ by $U_1=\Spec k[t]$ and $U_2=\Spec k[s]$ with gluing given by $s=\frac 1t$. 
Correspondingly, the cylinder $C$ is covered by $C_i=\pi^{-1}(U_i)$, $i=1,2$. 
We have $C_1=\Spec k[t][\frac Tt,\frac tT]$ and $C_2 = \Spec k[s][\frac Ts,\frac sT]$. 
Notice in particular that
\begin{equation*}
k[tT,T,\frac Tt] = k[t][\frac Tt,\frac tT] \cap k[s][\frac Ts, \frac sT],
\end{equation*}
so a degree-$-1$ derivation on $A$ is simultaneously a degree-$-1$ derivation on both $C_1$ and $C_2$.

Let us find the specific form of degree-$-1$ derivations on $C_1$. 
Given such a derivation $\delta$, we write $T\delta=\sigma'+\alpha'T\partial_T$ as in \S\ref{71}. 
Since $\delta$ is of degree $-1$, it must take the degree $n$ part of $k[t][\frac{T}{t},\frac{t}{T}]$ into its degree $n-1$ part. 
In other words,
\begin{equation*}
\delta(t^\ell(\frac{T}{t})^n) \in k[t]\cdot(\frac{T}{t})^{n-1},\quad\text{for all }n\in\ZZ. 
\end{equation*}

Using the decomposition of $\delta$, we get 
\begin{equation*}
\sigma'(t)t^{-1}\in k[t],\quad n\cdot (\frac{\alpha'}{t}-\frac{\sigma'(t)}{t^2})\in k[t], \quad\text{for all }n\in\ZZ.
\end{equation*}

The set of pairs $(\sigma'(t),\alpha')$ satisfying this condition is a free $k[t]$-module of rank $2$ with basis $(t,1)$ and $(t^2,0)$. 
Therefore, the $\O_{U_1}$-module $(\pi_*(\Theta_C))_{-1}\vert_{U_1}$ is free of rank $2$ with basis $\frac{t}{T}\partial_t + \partial_T$ and $\frac{t^2}{T}\partial_t$. 
Symmetrically, one can find that $(\pi_*(\Theta_C))_{-1}\vert_{U_2}$ is free with basis $\frac{s}{T}\partial_s + \partial_T$ and $\frac{s^2}{T}\partial_s$. 
How do these bases glue on $U_1 \cap U_2$? 
If $\ch(k) = 2$, we see that 
\begin{align*}
\frac{t}{T}\partial_t + \partial_T &= (\frac{t}{T}\partial_{t}s)\partial_s + \partial_T = -\frac{s}{T}\partial_s + \partial_T = \frac{s}{T}\partial_s + \partial_T,\\
\frac{t^2}{T}\partial_t &= \frac{t^2}{T}(\partial_{t}s)\partial_s = -\frac{1}{T}\partial_s = -s^{-2}(\frac{s^2}{T}\partial_s) = s^{-2}(\frac{s^2}{T}\partial_s).
\end{align*}
So the two pieces glue like $\O_X\oplus\O_X(-2)$.
If $\ch(k) \neq 2$, we can use the modified basis
\begin{equation*}
\frac{t}{T}\partial_t + \partial_T, \quad t(\frac{t}{T}\partial_t + \partial_T) -2\frac{t^2}{T}\partial_t = -\frac{t^2}{T}\partial_t + t\partial_T
\end{equation*}
on $U_1$, and correspondingly the basis
\begin{equation*}
\frac{s}{T}\partial_s + \partial_T, \quad-\frac{s^2}{T}\partial_s + s\partial_T
\end{equation*}
on $U_2$. 
Changing coordinate on $\PP^1_k$ from $t$ to $s$, we see that
\begin{align*}
\frac{t}{T}\partial_t + \partial_T &= -\frac{s}{T}\partial_s + \partial_T = s^{-1}(-\frac{s^2}{T}\partial_s + s\partial_T),\\
-\frac{t^2}{T}\partial_t + t\partial_T &= \frac{1}{T}\partial_s + \frac{1}{s}\partial_T = s^{-1}(\frac{s}{T}\partial_s + \partial_T). 
\end{align*}
So this time the two pieces glue like $\O_X(-1) \oplus \O_X(-1)$.

We summarize the computation as follows

\begin{prp}\label{89}
For $X=\PP^1_k$, $\L\cong\O_X(2)$, we have
\[
\M_{D,0}\cong
\begin{cases}
\O_X(2)\oplus\O_X,&\text{if }\ch(k)=2,\\
\O_X(1)\oplus\O_X(1),&\text{if }\ch(k)\neq2.
\end{cases}
\]
The generalized Euler sequence 
\begin{equation*}
0\to\L^d\to\M_{D,d}\to\Theta_{\PP^1}\otimes_{\O_X}\L^d\to0
\end{equation*}
is obtained by twisting the $0$th sequence by the $d$th power of $\L$.
\end{prp}

\begin{proof}
We showed above that $\M_{D,-1}\cong\O_X\oplus\O_X(-2)$ and $\M_{D,-1}\cong\O_X(-1)\oplus\O_X(-1)$ in the case that $\ch(k)=2$ or otherwise. 
The remaining argument follows from Remark~\ref{57}.
\end{proof}

\begin{rmk}
The splitting of the short exact sequence in Proposition~\ref{89} in case $\ch(k)=2$ was observed already by Wahl (see \cite[Rem.~(3.4)]{Wah76}).
It also follows from Proposition~\ref{94} since $\frac{ds}{s}-\frac{dt}{t}=2\frac{ds}{s}=0$.
\end{rmk}

\bibliographystyle{amsalpha}
\bibliography{dga}
\end{document}